\documentclass[12pt]{amsart}
\usepackage[top=1in, bottom=1.25in, left=1.25in, right=1.25in]{geometry}
\usepackage{amssymb,amsmath,amsthm,mathrsfs}
\usepackage{enumitem}

\usepackage{url}
\usepackage{cite}
\usepackage{amsfonts}
\usepackage{verbatim}
\usepackage{graphicx}
\usepackage{caption} 
\usepackage[all]{xy}
\usepackage{colonequals}
\usepackage{color}
\usepackage{ bbold }

\newcounter{constK}
\renewcommand{\theconstK}{{{\kappa}_{\arabic{constK}}}}
\newcommand{\constK}{\refstepcounter{constK}\theconstK}

\newcounter{constb}
\renewcommand{\theconstb}{{{b}_{\arabic{constb}}}}
\newcommand{\constb}{\refstepcounter{constb}\theconstb}

\def\R{{\mathbb{R}}}

\theoremstyle{plain}
\newtheorem{thm}{Theorem}[section]
\newtheorem*{thm*}{Theorem}
\newtheorem{defn}[thm]{Definition}
\newtheorem{prop}[thm]{Proposition}
\newtheorem{cor}[thm]{Corollary}
\newtheorem{lem}[thm]{Lemma}
\newtheorem{remark}[thm]{Remark}

\title[Effective equidistribution of large dimensional measures]{Effective equidistribution of large dimensional measures on affine invariant submanifolds}
\author{Anthony Sanchez}
\address{Department of Mathematics, University of California San Diego\\ 9500 Gilman Dr, La Jolla,
CA 92093, USA}
\email{ans032@ucsd.edu}

\begin{document}

\maketitle
\begin{abstract}

The unstable foliation, that locally is given by changing horizontal components of period coordinates, plays an important role in study of translation surfaces, including their deformation theory and in the understanding of horocycle invariant measures.

In this article we show that measures of large dimension on the unstable foliation equidistribute in affine invariant submanifolds and give an effective rate. An analogous result in the setting of homogeneous dynamics is crucially used in the effective density and equidistribution results of Lindenstrauss-Mohammadi and Lindenstrauss--Mohammadi--Wang.
\end{abstract} 
\section{Introduction}

The recent advances of Lindenstrauss--Mohammadi \cite{MR4549089} and Lindenstrauss--Mohammadi--Wang  \cite{LMW1,LMW2} mark a significant step forward in the quantitative behavior of orbits in homogeneous dynamics. Their works provide a potential framework for establishing effective results (such as effective density) in other settings including Teichm\"{u}ller dynamics. Their argument for this consists of three steps (see Mohammadi \cite{ICM} for an overview). While further investigation is needed to understand how the first two steps would translate in the setting of Teichm\"{u}ller dynamics, the main contribution of this article is that their final step can be successfully adapted in this context. Establishing effective density in Teichm\"{u}ller dynamics could shed light on a conjecture of Forni \cite{MR4297073} concerning the equidistribution of expanding horocyles in the moduli space of translation surfaces.

Broadly speaking, their closing step allows one to conclude effective equidistribution of expanding translates of subsets of the unstable foliation even when the dimension of the subsets is not full. This allows Lindenstrauss--Mohammadi \cite{MR4549089} to conclude effective density of horocycle flow for $\mathbb H^3$ and $\mathbb H \times \mathbb H^2$.  Their closing step is greatly inspired by one of Venkatesh \cite{MR2680486} where the argument was used to prove sparse equidistribution of horocycles. Some aspects that complicate the analysis in this setting include that the unstable foliation is not globally defined and the degeneration of certain norms in the thin part of moduli space.

We let $\mathcal M$ denote an \emph{affine invariant submanifold}. These are properly immersed smooth sub-orbifolds in the space of translation surfaces that are locally defined by real homogeneous linear equations in period coordinates. By the breakthrough work of Eskin--Mirzakhani \cite{MR3814652} and Eskin--Mirzakhani--Mohammadi \cite{MR3418528}, these can be equivalently described as $\mathrm{GL}^+(2, \mathbb R)$-orbit closures of translation surfaces. Each affine invariant submanifold supports an ergodic $\mathrm{SL}(2,\mathbb R)$-invariant probability measure that we will denote by $\mu_\mathcal M$. 

We set up notation needed for our main theorem and refer to Section \ref{Preliminaries} for formal definitions. Given a unit area $x\in\mathcal M$ and $r>0$, we will consider certain neighborhoods $B_r(x)$ of $\mathcal M$ that we call a \emph{period box} and we call the largest $r=r(x)$ with this property the \emph{injectivity radius} of $x$. Let 
$$\mathcal M_\eta = \{x\in \mathcal M:r(x)\ge\eta\}.$$

Let $a_t$ denote the action of Teichm\"{u}ller geodesic flow and $u_s$ denote the horocycle flow on $\mathcal M$. There is a foliation on $\mathcal M$ inherited from the linear structure of period coordinates that we will call the \emph{unstable foliation}. We choose this name because this foliation acts as the unstable foliation with respect to $a_t$. This foliation first appeared in Veech \cite{MR866707} where it was used to prove ergodicity of geodesic flow. Furthermore, the (non-horocyclic parts of the) unstable foliation were an important ingredient in the work on measure classification results for the horocycle flow \cite{MR2599885, MR4493581,CBF} for certain classes of translation surfaces. The classification of  measures invariant under the unstable foliation in special cases were considered in Lindenstrauss--Mirzakhani \cite{MR2424174} and Smillie--Smillie--Weiss--Ygouf \cite{SSWY}. Additionally, the unstable foliation was used to effectively count simple closed curves of hyperbolic surfaces in Eskin-Mirzakhani-Mohammadi\cite{MR4422208}. In short, the unstable foliation has proved fruitful in the moduli space of translation surfaces.

We denote the leaf of $x$ under this foliation as $W^{\mathrm{u}}(x)$. Let $B_r ^{\mathrm{u}}(x)\subset \mathcal M$ denote the connected component of $x$ in $B_r(x)\cap W^{\mathrm{u}}(x)$. Notice that the horocycle orbit of a point $x$ is contained inside of the leaf $W^{\mathrm{u}}(x)$; we let $W^{\mathrm{u}}(x)\cap H^{(0)}(x)$ denote the remaining directions and $B_r ^{\mathrm{u,0}}(x)$ denote the connected component of $x$ in $B_r(x)\cap W^{\mathrm{u}}(x)\cap H^{(0)}(x)$ the ball of radius $r$ of unstable, non-horocyclic directions. Let $\mu_x ^{\mathrm{u}}$ denote the leafwise measures of $\mathcal M$ along $W^{\mathrm{u}}(x).$ 

We will consider certain measures on $H^{(0)}$ that are ``large".

\begin{defn}
    Let $\varepsilon>0$ and $0<\delta<1$. Let $\rho$ be a probability measure on $B^{\mathrm{u,0}} _r(x)$ where $x\in \mathcal M$ and $0<r<r(x)$. We say $\rho$ is \emph{$\varepsilon$-rich at scale $\delta$} if for all $y\in B^{\mathrm{u,0}} _r(x)$, we have
    $$\rho(B_{\delta} ^{\mathrm{u,0}}(y))<\ref{b:rich}\delta^{d-\varepsilon}$$
    where $d$ is the dimension of $\mathcal M\cap W^{\mathrm{u}}(x)\cap H^{(0)}(x)$ and $\constb\label{b:rich}= \ref{b:rich}(\varepsilon,\delta)>0$.
\end{defn}

Our main result shows that we can replace the natural measure on unstable, non-horocylic directions with measures that are ``large" in the sense of the above definition and conclude effective equidistribution of expanding translates.

\begin{thm}\label{main}
 Let $\mathcal M$ be an affine invariant submanifold and suppose we had some period box $B_r(x)\subset \mathcal M$, where $x\in \mathcal M_\eta$ and $0<r<r(x)$, that supports a measure $\rho$ which is $\varepsilon$-rich at scale $\delta$. There exists a rate  $\constK\label{K:main}=\ref{K:main}(\mathcal M)$, a threshold $\varepsilon_0=\varepsilon_0(\kappa_1)>0$, and constants $L=L(\mathcal M)$ and $\constb\label{b:main}=\ref{b:main}(\mathcal M)$ such that if $\varepsilon<\varepsilon_0$, then for any $f\in C_c^\infty(\mathcal M)$ we have 
$$\left|\int_{B^{\mathrm{u,0}} _r(x)}\int_{s=0} ^1 f(a_tu_s y)\,ds\,d\rho(y)-\int_\mathcal M f \,d\mu_{\mathcal M}\right|<\ref{b:main} \mathcal C(f)\eta^{-L}e^{-\ref{K:main} t},$$
for $\frac{|\log(\delta)|}{8}\le t \le \frac{|\log(\delta)|}{4}$ and $0<\eta=\eta(t)\ll 1$. Here $\mathcal C$ is a norm on the set of compactly supported smooth functions of $\mathcal M$.

\end{thm}

Throughout, we use the notation $A \ll B$
 to mean that $A \le CB$ for some constant $C > 0$. 
\begin{remark}
The rate $\ref{K:main}$ in Theorem \ref{main} ultimately comes exponential mixing of Teichm\"{u}ller geodesic flow  of Avila-Gou\"{e}zel-Yoccoz \cite{MR2264836}.
\end{remark}

We shall see that despite Theorem \ref{main} giving effective equidistribution of measures of expanding translates of the form $ds\,d\rho$, we still need to use an effective equidistribution result of the leafwise measure $\mu ^{\mathrm{u}} _x$ on $W^{\mathrm{u}}(x)$. This is done in Section \ref{unstable} in Proposition \ref{Prop: EffUnstable}. A version of this is proved in Eskin-Mirzakhani-Mohammadi \cite{MR4422208}, but ours takes into account the cusp of moduli space. Additionally, the dependence of the cusp $\eta$ on time $t$ comes from the proof of Proposition \ref{Prop: EffUnstable}.

 \subsection*{Acknowledgements}
The author would like to thank Amir Mohammadi for many enlightening conversations. This work was supported by the National Science Foundation Postdoctoral Fellowship under grant number DMS-2103136.

\section{Preliminaries}\label{Preliminaries}
We review the basics of translation surfaces needed for this
article. For a detailed treatment of these topics, we refer the reader to the wonderful survey by Zorich \cite{MR2261104}. 
\subsection{On translation surfaces}
A \emph{translation surface} is a pair $x=(M,\omega)$ where $M$ is a compact, connected Riemann surface of genus $g$ and $\omega$ is a non-zero holomorphic 1-form on $M$. We can also, equivalently, view a translation surface as a union of finitely many polygons $P_1\cup P_2\cup\cdots\cup P_n$ in the Euclidean plane with gluings of parallel sides by translations such that for each edge there exists a parallel edge
of the same length and these pairs are glued together by a Euclidean translation. 

We denote the zeros of $\omega$ by $\Sigma$. By the Riemann-Roch theorem the sum of order of the zeros is $2g-2$ where $g$ denotes the genus of $M$. Thus, the space of genus $g$ translation surfaces can be stratified by integer partitions of $2g-2$. If $ \alpha = (\alpha_1,\ldots,\alpha_{|\Sigma|})$ is an integer partition of $2g-2$, we denote by $\mathcal H(\alpha)$ the moduli space of translation surfaces $\omega$ such that the multiplicities of the zeroes are given by $\alpha_1,\ldots,\alpha_{|\Sigma|}.$ 

There is a natural action by $\mathrm{SL}_2(\R)$ on the space of translation surfaces. This is most easily seen via the polygon definition: Given a translation surface $(M,\omega)$ that is a finite union of polygons $\{P_1,\ldots,P_n\}$ and $A\in \mathrm{SL}_2(\R)$, we define $A\cdot (M,\omega)$ to be the translation surface obtained by the union of the polygons  $\{AP_1,\ldots,AP_n\}$ with the same side gluings as for $\omega$. We are particularly interested in the Teichm\"{u}ller geodesic flow given by the action of
$$a_t=\begin{pmatrix}
    e^t&  0\\
    0 & e^{-t}
\end{pmatrix},$$
and the horocycle flow given by the action of
$$u_s=\begin{pmatrix}
    1&  s\\
    0 & 1
    \end{pmatrix}.$$

 Let  $\mathcal T \mathcal H (\alpha)$ denote the space of marked translation surfaces and $\pi:\mathcal T \mathcal H(\alpha)\to \mathcal H(\alpha)$ be the forgetful map that forgets the marking. The \emph{period map} $\Phi: \mathcal T \mathcal H(\alpha)\to H^1(M,\Sigma,\mathbb C)$ given by integrating over a basis of relative homology of a translation surface $x=(M,\omega)$ provides local coordinates for the space of marked translation surfaces. That is, given $2g+|\Sigma|-1$ curves $\gamma_1,\ldots,\gamma_{2g+|\Sigma|-1}$ that form a basis for $H_1(M, \Sigma,\mathbb Z)$, the period map is defined by
$$\Phi(x)=\left(\int_{\gamma_i}\omega\right)_{i=1} ^{2g+|\Sigma|-1}.$$
Consequently, period coordinates provide local coordinates for the space of translation surfaces $\mathcal H(\alpha)$.  Due to the splitting $$H^1(M, \Sigma,\mathbb C)=H^1(M, \Sigma,\mathbb R)\oplus H^1(M, \Sigma,\mathbb R),$$
we often write $\Phi(x)=a+ib$ for $a,b\in H^1(M, \Sigma,\mathbb R).$

For $x=(M,\omega)\in \mathcal H(\alpha)$, we define the \emph{tautological plane}, denoted by $E(x)$, to be the two dimensional subspace of $H^1(M,\Sigma, \mathbb R)$ spanned by the real and imaginary parts of $\omega$. Notice that the natural projection $p:H^1(M,\Sigma, \mathbb R)\to H^1(M, \mathbb R)$ defines an isomorphism between $E(x)$ and $p(E(x))\subset H^1(M,\mathbb R)$. Let $E(x)_\mathbb C \subset H^1(M,\Sigma,\mathbb C)$ denote the complexification of $E(x)$ and note that it is $\mathrm{SL}(2, \mathbb R)$-equivariant. Denote the symplectic compliment by 
$$H^{(0)}_\mathbb C(x)=\{c\in H^1(M,\Sigma,\mathbb C): p(c)\wedge p(E(x)_\mathbb C)\}.$$
Similarly, we define $H^{(0)}_\mathbb R(x)$ in a similarly manner and will use the simpler notation of $H^{(0)}(x)$. The superscript comes from the fact that, at the level of absolute homology, this subspace corresponds to cycles that have zero holonomy. 

We will need the non-divergence results of Athreya \cite{MR2247652}. See also Section 2.8 of Eskin-Mirzakhani-Mohammadi\cite{MR4422208}.

\begin{thm}
    There exists a compact subset $\mathcal M_0\subset \mathcal M$ and some $T_0=T_0(x)>0$ with the following property. For every $t_0$ and every $x\in \mathcal M$, there exists $s\in[0,1/2]$ and $t_0\le t\le T_0 $ such that $a_tu_s x\in \mathcal M_0$.
\end{thm}


\subsection{The AGY-norm and period boxes}
We will utilize the norm defined in Avila-Gou\"{e}zel-Yoccoz \cite{MR2264836}. For $x=(M,\omega)\in \mathcal M$ and any $c\in H^1(M,\Sigma,\mathbb C)$, we define
$$\|c\|_{\mathrm{AGY},x}=\sup_\gamma\frac{|c(\gamma)|}{|\int_\gamma \omega|}$$
where $\gamma$ is a saddle connection of $x$. It was shown in Avila-Gou\"{e}zel-Yoccoz \cite{MR2264836} that this defines a norm and the corresponding Finsler metric is complete. 

We will often use the following lemma that allows us to compare the norm at points on that differ by certain elements of $\mathrm{SL}_2(\mathbb R)$. See also Lemma 2.4 of Eskin-Mirzakhani-Mohammadi\cite{MR4422208}.

\begin{lem}[Lemma 5.2, \cite{MR2264836}]\label{lemmaAGY}
    For $c\in H^1(M,\Sigma,\mathbb C)$, $t\ge0$, and $s\in [0,1]$ we have
    $$e^{-2-2t}\|c\|_{\mathrm{AGY},x}\le\|(a_tu_s)_* c\|_{\mathrm{AGY, } a_tu_sx}\le e^{2+2t}\|c\|_{\mathrm{AGY},x}.$$
\end{lem}

For $r>0$ and $ x \in \mathcal {TH}(\alpha)$, we define 

$$R_r(x):=\{\Phi(x) + a' +ib':a',b'\in H^1(M,\Sigma,\mathbb R),\|a'+ib'\|_{\mathrm{AGY},x}\le r\}.$$

By non-divergence results of the horocycle flow $u_t$  and the construction of a Margulis function on the space of translation surfaces (See Lemma 2.6 of Eskin-Mirzakhani-Mohammadi\cite{MR4422208} for details), for any $x\in \mathcal H(\alpha)$ and any lift $\tilde x\in \mathcal {TH}(\alpha)$, there exists $r(x)$ such that for all $0<r<r(x)$, the restriction of the covering map $\pi$ to $$B_r(\tilde x):= \Phi^{-1}(R_r(\tilde x))$$
is injective. 

In this case (i.e. for $0<r\le r(x)$), we call $B_r(x) = \pi(B_r(\tilde x))$ a \emph{period box} of radius $r$ centered at $x$ and  we call $r(x)$ the \emph{injectivity radius} of $x$. Additionally, when we want to work with $x$ in an affine invariant manifold $\mathcal M$, we consider $B_r ^\mathcal M(x) := B_r(x)\cap \mathcal M$  and continue to call it the period box of radius $r$ centered at $x.$ In fact, if we write ``$x\in\mathcal M$" we will only work with $B_r ^\mathcal M(x)$ and often suppress the superscript.

For convenience, we will also work with the following max norm,
$$\|c\|_{\text{max},x} = \max_i |\lambda_i|\|c_i\|_{\mathrm{AGY},x}$$
where $c$ is written as a linear combination $\sum_i \lambda_i c_i$ of any basis $c_i$ of $H^1(M,\Sigma;\mathbb C)$. Later we will allow the parameter $\eta$ in $\mathcal M_\eta$ to depend on the time $t$ and so we show that the max norm and the AGY norm are comparable on $\mathcal M_\eta$.

\begin{lem}\label{lemmamax}
        For $c\in H^1(M,\Sigma,\mathbb C)$, and $x\in \mathcal M_\eta$, there exists $n>0$ such that
    $$\eta^n\|c\|_{\mathrm{max},x}\le\| c\|_{\mathrm{AGY, } x}\le (2g+|\Sigma|-1)\|c\|_{\mathrm{max},x}$$
    where the implicit constant is absolute.
\end{lem}

\begin{proof}
    Let $c=\sum_i \lambda_i c_i\in H^1(M,\Sigma,\mathbb C)$ and $x\in \mathcal M_\eta$ where $\eta = e^{-tb}$ for some $b\in(0,1)$.

    By definition of $\|\cdot\|_{\textrm{max},x}$, we have
    $$\|c\|_{\textrm{AGY},x}\le (2g+|\Sigma|-1)\|c\|_{\textrm{max},x}.$$

    Additionally, there exists some absolute constant $\ell$ such that
    $$\|c\|_{\textrm{AGY},x'}\ge \ell \|c\|_{\textrm{max},x'}$$
    for any $x'\in\mathcal M_0$ for the fixed compact set $\mathcal M_0$ from the statement of non-divergence. Furthermore, by non-divergence, there exists $s\in[0,1/2]$ and $t\le t'$ such that $x':=a_{t'}u_sx \in \mathcal M_0$. Then by using Lemma \ref{lemmaAGY} twice we deduce,     $$e^{2t'+2}\|c\|_{\textrm{AGY},x}\ge\|c\|_{\textrm{AGY},x'}\ge \ell \|c\|_{\textrm{max},x'}\ge \ell e^{-2t'-2} \|c\|_{\textrm{max},x}$$
    for any $x\in \mathcal M_\eta.$
    Thus, $$\|c\|_{\textrm{AGY},x}\ge \ell e^{-4t'-4} \|c\|_{\textrm{max},x}.$$
    So by choosing $n\ge(4t' + 4 - \log (\ell))/bt$, we get 
    $$\|c\|_{\textrm{AGY},x}\ge e^{-nbt} \|c\|_{\textrm{max},x}=\eta^n \|c\|_{\textrm{max},x}.$$
\end{proof}

In particular, balls with respect to AGY and the max norm are comparable,
$$B_{\eta^n}\subset B_r ^{\textrm{max}}\subset B_{(2g+|\Sigma|-1)r}.$$
This allows us to compare measures of one type of ball with measures of the other. The multiplicative constant that appears in the lower bound when considering norms on $\mathcal M_\eta$ requires very careful analysis because it changes as $t$ changes.
\subsection{The unstable foliation}
We follow Avila-Gou\"{e}zel \cite{MR3071503} section 4. Given a point $x=(M,\omega)\in\mathcal M$, the tangent space $T_x\mathcal M\simeq H^1(M,\Sigma,\mathbb C)$ decomposes as
$$T_x\mathcal M =\mathbb R\textbf{v(x)}\oplus E^{\mathrm{u}} (x)\oplus E^s(x)$$
where $\textbf{v}(x)$ determines the direction of Teichm\"{u}ller geodesic flow and has $\|\textbf{v}(x)\|_{\mathrm{AGY},x}=1$, and
$$E^{u}(x)=T_x\mathcal M \cap D\Phi_x^{-1}(H^1(M,\Sigma,\mathbb R)),$$
$$E^{s}(x)=T_x\mathcal M \cap D\Phi_x^{-1}(\textbf{i}H^1(M,\Sigma,\mathbb R)).$$
We call $E^{\mathrm{u}}(x)$ (respectively, $E^{\mathrm{s}}(x)$) the \emph{unstable} (respectively, the \emph{stable}) manifold.

Proposition 4.4 of Avila-Gou\"{e}zel \cite{MR3071503} shows that the subspaces $E^{\mathrm{u,s}}(x)$ depend smoothly on $x$ and are integrable. We denote the corresponding leaves by $W^{\mathrm{u}}(x)$ and $W^{\mathrm{s}}(x)$, respectively. We call $W^{\mathrm{u}}(x)$ the \emph{unstable foliation} and $W^{\mathrm{s}}(x)$ the \emph{stable foliation}. These foliations are well defined on affine invariant submanifolds $\mathcal M$ even when we restrict to surfaces of unit area by the paragraph proceeding Definition 3.4 of \cite{SSWY}. Going forward, we assume that we are working with the foliation on unit area surfaces and by abuse of notation, we continue to denote the leaves by $W^{\mathrm{u}}(x)$.

In the literature, $W^{\mathrm{u}}(x)$ has also been the called the horospherical foliation or horizontal foliation. Our choice for the name of $W^{\mathrm{u}}(x)$ is due to the fact that this foliation acts as the unstable foliation with respect to the Teichm\"{u}ller geodesic flow. For more on this foliation see  Smillie--Smillie--Weiss--Ygouf \cite{SSWY} and Eskin-Mirzakhani-Mohammadi\cite{MR4422208}. 
 
Two points in period coordinates 
$$z=\begin{pmatrix}
    x_1&\cdots & x_n\\
    y_1&\cdots &y_n
\end{pmatrix} \text{ and }z'=\begin{pmatrix}
    x_1 '&\cdots & x_n'\\
    y_1'&\cdots &y_n'
\end{pmatrix}$$ that are in the same chart will be in the same leaf $W^{\mathrm{u}}(x)$ if they differ by some $w = \begin{pmatrix}
    w_1&\cdots & w_n\\
    0&\cdots &0
\end{pmatrix}.$
This way it is easy to see that the horocycle flow preserves the leaf and contributes one dimension to the dimension of the leaf. More succinctly, if we let $\Phi(x)=a+ib$. Then, the unstable leaf $W^{\mathrm{u}}(x)$ is locally identified with $\Phi(x)+sb+w$ for $s\in \mathbb R$ and  $w\in H^{(0)}(x)$. 

 Let $\mu_x ^{\mathrm{u,s}}$ denote the leafwise measures of $\mu_\mathcal M$ along $W^{\mathrm{u,s}}(x).$ Since we work with $\textrm{SL}(2,\mathbb R)$-invariant measures, we have $\mu_x ^{\mathrm{u,s}}$ are simply the Lebesgue measure on the leaf for a.e. $x\in\mathcal M$ (Theorem 2.1 of Eskin--Mirzakhani \cite{MR3814652}). If $B_r(x)$ is a period box centered at $x\in\mathcal M$, then $\mu_\mathcal{M}|_{B_r(x)}$ has a product structure of $\,d\text{Leb}\times \,d\mu_x ^{\mathrm{u}}\times \,d\mu_x ^{\mathrm{s}}.$

 Let $B_r ^{\mathrm{u}}(x)\subset \mathcal M$ denote the connected component of $x$ in $B_r(x)\cap W^{\mathrm{u}}(x)$. As mentioned in the introduction, the horocycle orbit of a point $x$ is contained inside of the leaf $W^{\mathrm{u}}(x)$; we let $B_r ^{\mathrm{u,0}}(x)$ denote the connected component of $x$ in $B_r(x)\cap W^{\mathrm{u}}(x)\cap H^{(0)}(x)$. 

\subsection{Smooth structure on affine manifolds}
Following Avila-Gou\"{e}zel \cite{MR3071503}  and Eskin-Mirzakhani-Mohammadi\cite{MR4422208} we endow affine invariant manifolds with a smooth structure. For a function $\varphi$ defined on an affine invariant manifold $\mathcal M$, define
$$c_k(\varphi)=\sup |D^k\varphi(x,v_1,\ldots,v_k)|$$
where the supremum is taken over $x$ in the domain of $\varphi$ and $v_1,\ldots,v_k\in T_x\mathcal M$ with AGY-norm at most 1. We define the \emph{$C^k$-norm} of $\varphi$ to be $\|\varphi\|_{C^k}=\sum_{j=0}^kc_j(\varphi)$. 

We denote the space of compactly supported functions with finite $C^k$ norm on $\mathcal M$ by $C_c^{k}(\mathcal M)$ and define $C_c^{\infty}(\mathcal M)$ similarly.

In this article we will only need the $C^1$-norm of functions and we simplify our notation by defining $\mathcal C( \varphi):=\|\varphi\|_{C^1}$.  Note that by Lemma \ref{lemmaAGY}, we have $\mathcal C (f\circ (a_tu_s))\le e^{2+2t}\mathcal C(f)$ for $t\ge 0$ and $s\in[0,1]$.

Additionally, we state a result from Eskin-Mirzakhani-Mohammadi\cite{MR4422208} that allows us to replace characteristic function with smooth approximations. 

Let $W$ denote one of $\mathcal M$ or $ \mathcal M\cap W^{\mathrm{u}}(x)$ for $x\in \mathcal M$. Let $E$ be a compact subset of $W$ and $r(E)=\inf\{r(x):x\in E\}$. For $0<\varepsilon< r(E)/10$, we define the following open neighborhood of $W$,
$$E_{\varepsilon} ^W=\{y\in W: r(y)\ge \varepsilon \text{ and }B_\varepsilon(y)\cap E\ne \emptyset\}.$$
In practice, we will take $E=\mathcal M_\eta = \{x\in\mathcal M:r(x)\ge \eta\}$.

Let $r>0$ and $L>0$. Let $\mathcal{S}_W(E,r,L)$ denote the class of Borel functions $0\le f\le 1$ supported and defined everywhere in $E$ with the following properties: For $\varepsilon\le r/(10L)$ there exists $\varphi_{+,\varepsilon}$, $\varphi_{-,\varepsilon}
\in C_c(E_\varepsilon ^ W)$ such that
\begin{enumerate}
    \item $\varphi_{-,\varepsilon}\le f \le \varphi_{+,\varepsilon}$,
    \item $\mathcal C(\varphi_{\pm,\varepsilon})\le \varepsilon^{-L}$, and
    \item $\|\varphi_{+,\varepsilon}-\varphi_{-,\varepsilon}\|_2\le \varepsilon^{1/2}\|f\|_2$.
\end{enumerate}

We need the following result from Eskin-Mirzakhani-Mohammadi\cite{MR4422208} that allows us to replace characteristic functions with smooth approximations.
\begin{lem}[Lemma 2.11, \cite{MR4422208}]
There exists some $L$ depending only on $\mathcal M$ such that for any $0<r\le r(x)$,
$$1_{B_r ^{\mathrm{u}}(x)}\in \mathcal{S}_{B_r ^{\mathrm{u}}(x) }(E,r,L)\text{ and } 1_{B_r(x)} \in \mathcal{S}_{B_r(x)} (E,r,L).$$
\end{lem}

We fix one such $L$ so that the above lemma holds and drop the dependence on $L$. Additionally, we drop the dependence on $W$ when the context is clear and similarly for $E$ when the compact set is not relevant except that it is a compact set containing $x$.

\subsection{Decay of correlations}

We need the following two results on decay of correlations.
\begin{thm}[Exponential mixing\cite{MR2264836}]
    Let $(\mathcal M,\mu_\mathcal M)$ be an affine invariant submanifold. There exists a positive constant $\kappa'=\kappa'(\mathcal M,\mu_\mathcal M)$ such that if $\phi,\psi\in C_c^\infty (\mathcal M)$, then 
    $$\left|\int\phi(a_tx)\psi(x)d\mu_\mathcal M (x)- \int \phi\,d\mu_\mathcal M\int\psi\,d\mu_\mathcal M\right|\ll \mathcal C(\phi)\mathcal C(\psi)e^{-\kappa't}$$
    where the implied constants depend on $(\mathcal M,\mu_\mathcal M)$.
\end{thm}

We equip $\mathrm{SL}(2,\mathbb R)$ with the right-invariant metric $d$ induced by the Killing form on it's Lie algebra. Recall that we choose 
$$a_t=\begin{pmatrix}
    e^t&  0\\
    0 & e^{-t}
\end{pmatrix},$$
and so $d(a_t,e)=2t$. By using the Cartan decomposition of $\mathrm{SL}(2,\mathbb R)$ and invariance of $\mathcal M$ under $\mathrm{SL}(2,\mathbb R)$, we obtain the following corollary (see also Proposition B.2 of \cite{MR2264836}).

\begin{cor}\label{Cor:Decay of Cor}
    Let $(\mathcal M,\mu_\mathcal M)$ be an affine invariant submanifold. There exists a positive constant $\kappa_\mathcal M=\kappa_\mathcal M(\mathcal M,\mu_\mathcal M)$ such that if $\phi,\psi\in C_c^\infty (\mathcal M)$, then 
    $$\left|\int\phi(gx)\psi(x)\,d\mu_\mathcal M (x)- \int \phi \,d\mu_\mathcal M\int\psi \,d\mu_\mathcal M\right|\ll \mathcal C(\phi)\mathcal C(\psi)e^{-\kappa_\mathcal M d(e,g)}$$
    where the implied constant depends on $(\mathcal M,\mu_\mathcal M)$.
\end{cor}

\section{Effective Equidistribution of the unstable foliation}\label{unstable}
In this section we verify that the effective equidistribution of the unstable foliation always holds in the setting of Teichm\"{u}ller dynamics. We need the following result of Eskin-Mirzakhani-Mohammadi \cite{MR4422208} that relies on the exponential mixing of Teichm\"{u}ller geodesic flow Avila-Gou\"{e}zel-Yoccoz \cite{MR2264836}.

\begin{prop}[Prop 3.2, \cite{MR4422208}]\label{Prop 3.2} Let $\mathcal M$ be an affine invariant submanifold. There exists $\constK\label{K:EMM}$ depending only on $\mathcal M$ with the following property. Let $x\in \mathcal M$, $0<r\le r(x)$, and let $B_r(x)$ be  a period box centered at $x$. Let $\psi^{\mathrm{u}}\in C_c ^\infty (B_r ^{\mathrm{u}}(x))$. Then for any $\phi\in C_c^\infty (\mathcal M)$ we have
$$\left|\int_{W^{\mathrm{u}}(x)}\phi(a_t y)\psi^{\mathrm{u}} (y)\,d\mu_x ^{\mathrm{u}}(y)-\int_{\mathcal M}\phi \,d\mu_\mathcal M \int _{W^{\mathrm{u}}(x)}\psi^{\mathrm{u}} d\mu_x ^{\mathrm{u}}(y)\right |\le \mathcal C (\phi)\mathcal C (\psi^{\mathrm{u}})e^{-\ref{K:EMM} t}.$$
\end{prop}

We use the above and an approximation argument to obtain effective equidistribution of the unstable foliation. 

\begin{prop}\label{Prop: EffUnstable}
 Let $\mathcal M$ be an affine invariant submanifold.  There exists $\constK\label{K:effec}$ depending only on $\mathcal M$ with the following property. Let $x\in \mathcal M_\eta$, $0<r\le r(x)$, and let $B_r(x)$ be  a period box centered at $x$. Then, there exists $L>0$, and $\constb\label{b:effec}=\ref{b:effec}(L)$ such that for any $f\in C_c^\infty(\mathcal M)$ we have 
 $$\left|\frac{1}{\mu^{\mathrm{u}} _x (B^{\mathrm{u}} _r(x))}\int_{B_r ^{\mathrm{u}}(x)}f(a_t y)\,d\mu^{\mathrm{u}} _x(y)-\int_{\mathcal M }f(x)\,d\mu_\mathcal M(x)\right|<\ref{b:effec}\mathcal C(f)\eta^{-L}e^{-\ref{K:effec} t}$$
 where $t>0$ and $0<\eta=\eta(t)\ll 1$.
    \end{prop}
\begin{proof}
By Lemma 2.1, we can find functions that approximate $1_{B_r ^{\mathrm{u}}(x)}$ that are well behaved in that they satisfy 
\begin{enumerate}
    \item $\varphi_{-,\varepsilon}\le 1_{B_r ^{\mathrm{u}}(x)} \le \varphi_{+,\varepsilon}$,
    \item $\mathcal C(\varphi_{\pm,\varepsilon})\le \varepsilon^{-L}$, and
    \item $\|\varphi_{+,\varepsilon}-\varphi_{-,\varepsilon}\|_2\le \varepsilon^{1/2}\|1_{B_r ^{\mathrm{u}}(x)}\|_2$
\end{enumerate}
for any $\varepsilon<\eta/10L.$

Without loss of generality, we suppose that $$\int_{B_r ^{\mathrm{u}}(x)}f(a_t y)\,d\mu^{\mathrm{u}} _x(y)-\int_{\mathcal M }f(x)\,d\mu_\mathcal M(x)\cdot\mu^{\mathrm{u}} _x (B^{\mathrm{u}} _r(x))>0.$$ 
Then,
\begin{align*}
&\int_{B_r ^{\mathrm{u}}(x)}f(a_t y)\,d\mu^{\mathrm{u}} _x(y)-\int_{\mathcal M }f\,d{\mu_\mathcal M}\cdot\mu^{\mathrm{u}} _x (B^{\mathrm{u}} _r(x))\\
&\le\int_{W ^{\mathrm{u}}(x)}f(a_t y)\varphi_{+,\varepsilon}(y)\,d\mu^{\mathrm{u}} _x(y)-\int_{\mathcal M }f\,d\mu_\mathcal M\cdot\mu^{\mathrm{u}} _x (B^{\mathrm{u}} _r(x))\\
&=\int_{W ^{\mathrm{u}}(x)}f(a_t y)\varphi_{+,\varepsilon}(y)\,d\mu^{\mathrm{u}} _x(y)-\int_{\mathcal M}f \int_{W ^{\mathrm{u}}(x)}\varphi_{+,\varepsilon}+\int_{\mathcal M}f \int_{W ^{\mathrm{u}}(x)} \varphi_{+,\varepsilon}-\int_{\mathcal M }f\cdot\mu^{\mathrm{u}} _x (B^{\mathrm{u}} _r(x))\\
&= \left(\int_{W ^{\mathrm{u}}(x)}f(a_t y)\varphi_{+,\varepsilon}(y)\,d\mu^{\mathrm{u}} _x(y)-\int_{\mathcal M}f \int \varphi_{+,\varepsilon}\right)+\int_{\mathcal M}f\left( \int_{W ^{\mathrm{u}}(x)} \varphi_{+,\varepsilon}-1_{B_r ^{\mathrm{u}}(x)}\right)\\
&= \left(\int_{W ^{\mathrm{u}}(x)}f(a_t y)\varphi_{+,\varepsilon}(y)\,d\mu^{\mathrm{u}} _x(y)-\int_{\mathcal M}f \int \varphi_{+,\varepsilon}\right)+\int_{\mathcal M}f\left( \int_{W ^{\mathrm{u}}(x)} \varphi_{+,\varepsilon}-\varphi_{+,\varepsilon}\right)
\end{align*}

A similar computation shows that we can achieve an analogous lower bound. Hence, by applying Proposition \ref{Prop 3.2} on the first term and Cauchy-Schwartz on the second yields
\begin{align*}
|\int_{B_r ^{\mathrm{u}}(x)}f(a_t y)\,d\mu^{\mathrm{u}} _x(y)-\int_{\mathcal M }f\,d&\mu\cdot\mu^{\mathrm{u}} _x (B^{\mathrm{u}} _r(x))|= \mathcal C(f)\mathcal C(\varphi_{+,\varepsilon})e^{-\ref{K:EMM}t}+\|f\|_2\|\varphi_{+,\varepsilon}-\varphi_{-,\varepsilon}\|\\
&\le \mathcal C(f)\mathcal (\mathcal C(\varphi_{+,\varepsilon})+\mathcal C(\varphi_{-,\varepsilon}))e^{-\ref{K:EMM}t}+\|f\|_2\|\varphi_{+,\varepsilon}-\varphi_{-,\varepsilon}\|.
\end{align*}

Applying properties 2 and 3 of the functions from Lemma 2.1 and observing that we are on a probability space yields
$$\left|\int_{B_r ^{\mathrm{u}}(x)}f(a_t y)\,d\mu^{\mathrm{u}} _x(y)-\int_{\mathcal M }f\,d\mu\cdot\mu^{\mathrm{u}} _x (B^{\mathrm{u}} _r(x))\right|\le 2\mathcal C(f)\varepsilon^{-L}e^{-\ref{K:EMM}t}+\mathcal C (f)\varepsilon^{1/2}\mu^{\mathrm{u}} _x (B^{\mathrm{u}} _r(x)).
$$

Now choosing $\varepsilon=\eta/20L$ and only considering $\eta\le e^{-t\ref{K:EMM}/2}$,  we obtain
\begin{align*}
\left|\int_{B_r ^{\mathrm{u}}(x)}f(a_t y)\,d\mu^{\mathrm{u}} _x(y)-\int_{\mathcal M }f\,d\mu\cdot\mu^{\mathrm{u}} _x (B^{\mathrm{u}} _r(x))\right|&\le 2\mathcal C(f)(20L)^L e^{-t\ref{K:EMM}}\eta^{-L}+\mathcal C(f)\eta^{1/2}\mu^{\mathrm{u}} _x (B^{\mathrm{u}} _r(x))\\&\le2(20L)^L \mathcal  C(f) e^{-t\ref{K:EMM}}\eta^{-L}.
\end{align*}

Hence, the proof is completed by choosing $\ref{K:effec}=\ref{K:EMM}$ and $C=2(20L)^L.$

\end{proof}

 For the rest of the article, we will always assume $\eta\le e^{-t\ref{K:EMM}/2}$.

\section{Proof of main theorem}

In this section we prove Theorem \ref{main}. 
The argument is inspired by Venkatesh \cite{MR2680486}. Broadly speaking, we will thicken the measure $\rho$ to one on the full unstable foliation using a certain function $\varphi$ that will be defined in the course of the proof. The dimension condition of $\rho$ implies that the $L^\infty$-norm of the thickened integral is not too large. We take an extra average in the horocyclic direction and take the square of this thickening. We complete the proof by using the decay of correlations of $u_s$ to show terms of the square far from the diagonal are controlled and terms close to the diagonal have small measure and are negligible.
\begin{proof}
We make a number of simplifications for ease of exposition. For example, we only deal with the ball of radius 1, $B_1 ^{\mathrm{u,0}}(x)$, but the proof works with small modifications for any ball. Additionally, without loss of generality, we assume $\int_{\mathcal M} f d\mu_\mathcal{M}= 0$ and that $f$ is a Lipschitz function. Note that by definition of $\mathcal C(\cdot)$, the Lipschitz norm is dominated by $\mathcal{C}(f)$.

Furthermore, it is convenient to work with the max norm $\|\cdot\|_{\textrm{max},x}$ on $H^1(M,\Sigma,\mathbb C)$. By Lemma \ref{lemmamax}, the AGY norm and max norm are comparable on $\mathcal M_\eta$.

Let $\Phi(x)=a+ib$. 
Let $d$ denote the dimension of $\mathcal M \cap W^{\mathrm{u}}(x)\cap H^{(0)}$ and let $w_1,\ldots,w_d$ be linearly independent cohomology classes that span $\mathcal M \cap W^{\mathrm{u}}(x)\cap H^{(0)}$ and $\mathbf{w}=(w_1\ldots,w_d)$. Recall, that the unstable leaf $W^{\mathrm{u}}(x)$ is locally identified with $\Phi(x)+sb+w$ for $w\in H^{(0)}(x)$. 

By choosing $N\in\mathbb N$ so that $\frac{1}{N}\le \delta<\frac{1}{N-1}$, we suppose that $\delta=1/N$. 

We will only work with the non-negative orthant of $B_1 ^{\mathrm{u,0}}(x)$ since considering the remaining parts only result in multiplication by a fixed multiplicative constant. For a vector $\mathbf{k}\in\{0,\ldots, N-1\}^d$, define the $\delta$-ball $I_\mathbf{k}$ set of directions one can move in $W^{\mathrm{u}}(x)\cap H^{(0)}(x)$ by
$$I_\mathbf{k}=\{\Phi(x) +\mathbf{r}\cdot \mathbf{w}\in B_1 ^{\mathrm{u,0}}(x): \mathbf{r}=(r_1,\ldots,r_d), r_j\in [k_j\delta,(k_j+1)\delta), \text{ for each }j\}.$$
Define $c_\mathbf{k}=\rho(I_\mathbf{k})$. Notice that since the $\delta$-balls $(I_\mathbf{k})_{\mathbf{k} \in\{0,\ldots, N-1\}^d}$ are disjoint and their union is $B_1 ^{\mathrm{u,0}}(x)$ that $\sum_{\mathbf{k} \in\{0,\ldots, N-1\}^d}c_\mathbf{k}=1$. 

For each $\mathbf{k} \in\{0,\ldots, N-1\}^d$, let
$$B_\mathbf{k} =\left\{\Phi(x) + s b + \mathbf{r}\cdot \mathbf{w}\in B_1 ^{\mathrm{u}}(x):0\le s\le 1,r_j\in\left(k_j\delta,k_j\delta+\frac{\delta}{4}\right)\text{ for each }j \right\}$$
denote a thickening to the full unstable leaf $W^{\mathrm{u}}(x)$.
The sets $B_\mathbf{k}$ continue to be disjoint for different indices $\mathbf{k}$.

Then,
\begin{align*}
    \int_{B^{\mathrm{u,0}}(x)}\int_{s=0} ^1 f(a_tu_s y)\,ds\,d\rho(y)&=\sum_\mathbf{k}\int_{I_\mathbf{k}}\int_{s=0} ^1 f(a_t(\Phi(x)+sb+\mathbf{r}\cdot\mathbf{w}))\,ds\,d\rho(\mathbf{r})
    \\&=\sum_\mathbf{k}\int_{\mathbf{s}\in[0,\delta)^d}\int_{s=0} ^1 f(a_tu_s(\Phi(x)+(\delta\mathbf{k}+\mathbf{r})\cdot\mathbf{w}))\,ds\,d\rho(\mathbf{r}).
\end{align*}
Thus,
\begin{align*}
    &\left|\int_{B^{\mathrm{u,0}}(x)}\int_{s=0} ^1 f(a_tu_s y)\,ds\,d\rho(y) - \sum_\mathbf{k}c_{\mathbf{k}}\int_{s=0} ^1 f(a_t(\Phi(x)+sb+\delta\mathbf{k}\cdot\mathbf{w}))\,ds\right|\\
    &\le\sum_\mathbf{k}\int_{I_\mathbf{k}}\int_{s=0} ^1 \left|f(a_t(\Phi(x)+sb+\mathbf{r}\cdot\mathbf{w})) - f(a_t(\Phi(x)+sb+\delta\mathbf{k}\cdot\mathbf{w}))\right|\,ds\,d\rho(\mathbf{r})\\
    = &\sum_\mathbf{k}\int_{I_\mathbf{k}}\int_{s=0} ^1 \left|f(a_tu_s(\Phi(x)+\mathbf{r}\cdot\mathbf{w})) - f(a_tu_s(\Phi(x)+\delta\mathbf{k}\cdot\mathbf{w}))\right|\,ds\,d\rho(\mathbf{r}).
\end{align*}

We now compare the difference between $f(a_tu_s(\Phi(x)+\mathbf{r}\cdot\mathbf{w}))$ and $ f(a_tu_s(\Phi(x)+\delta\mathbf{k}\cdot\mathbf{w}))$ by recalling that we assume $f$ is a Lipschitz function and that on the domain of integration, the difference between the inputs of $f$ is given by $\|a_tu_s(\mathbf{r'}\cdot\mathbf{w})\|_{\textrm{AGY},a_tu_sx}$ where $\mathbf{r'}\in[0,\delta)^d$. Additionally, we utilize Lemma \ref{lemmaAGY}, Lemma \ref{lemmamax}, and that $t\le |\log(\delta)|/4$, to obtain
\begin{align*}
\left|f(a_tu_s(\Phi(x)+\mathbf{r}\cdot\mathbf{w})) - f(a_tu_s(\Phi(x)+\delta\mathbf{k}\cdot\mathbf{w}))\right|&\le \mathcal C(f)\|a_tu_s(\mathbf{r'}\cdot\mathbf{w})\|_{\textrm{AGY},a_tu_sx}\\
&\ll C(f) e^{2+2t} \|\mathbf{r'}\cdot\mathbf{w}\|_{\textrm{max},x}\\
&\ll C(f)\delta^{-1/2} \delta = C(f) \delta^{1/2}.
\end{align*}
Thus, we obtain 
\begin{align*}
&\left|\int_{B^{\mathrm{u,0}}(x)}\int_{s=0} ^1 f(a_tu_s y)\,ds\,d\rho(y) - \sum_\mathbf{k}c_{\mathbf{k}}\int_{s=0} ^1 f(a_t(\Phi(x)+sb+\delta\mathbf{k}\cdot\mathbf{w}))\,ds\right|\\
\\&\ll \sum_\mathbf{k}\int_{I_\mathbf{k}}\int_{s=0} ^1 \mathcal C (f)\delta^{1/2}\,ds\,d\rho(\mathbf{r})=\mathcal C (f)\delta^{1/2}\sum_{\mathbf{k}}\rho(I_k)=\mathcal C (f)\delta^{1/2}
\end{align*}
and so it suffices to understand the behavior of $f$ on the discrete points $\delta\mathbf{k}\cdot \mathbf{w}.$

Let 
$$\varphi = \sum_{\mathbf{k}} \mu_x ^{\mathrm{u}}(B_\mathbf{k})^{-1} \cdot c_\mathbf{k}\mathbb{1}_{B_\mathbf{k}}.$$

By Lemma \ref{lemmamax} and since $\rho$ is $\varepsilon$-rich at scale $\delta$, we have $$\frac{\rho (I_\mathbf{k})}{\mu_x ^\mathrm{u}(B_\mathbf{k})}\le \frac{(2g+|\Sigma|-1)^2 b_2\delta^{d-\varepsilon}}{\eta^{nd}4^{-d}\delta^{d}}\ll\eta^{-nd}\delta^{-\varepsilon}$$
where the implied constant is absolute.

Thus, we have the following pointwise bound $$|\varphi|\ll\eta^{-nd}\delta^{-\varepsilon} \sum_{\mathbf{k}}\mathbb{1}_{B_\mathbf{k}}=\eta^{-nd}\delta^{-\varepsilon} \mathbb{1}_{\cup_\mathbf{k} B_\mathbf{k}}\eta^{-nd}\delta^{-\varepsilon}\le \eta^{-nd}\delta^{-\varepsilon}$$
by the disjointness of $B_{\mathbf{k}}$. 

We have, by noting that $\,d\mu^{\mathrm{u}} _x = \,d\mathbf{r}\,ds$,
\begin{align*}
  &\left|\sum_\mathbf{k}c_{\mathbf{k}}\int_{s=0} ^1 f(a_tu_s(\Phi(x)+\delta\mathbf{k}\cdot\mathbf{w}))\,ds-\int_{W^{\mathrm{u}}(x)}\varphi(y)f(a_t y)\,d\mu^{\mathrm{u}} _x(y)\right|\le\\
  &\sum_\mathbf{k}4^d\delta^{-d}\cdot c_{\mathbf{k}} \int_{B_\mathbf{k}}\left|f(a_tu_s(\Phi(x)+\delta\mathbf{k}\cdot\mathbf{w}))-f(a_tu_s(\Phi(x)+\mathbf{r}\cdot\mathbf{w}))\right|\,d\mathbf{r}\,ds\\
    &\ll \mathcal C (f) \delta^{1/2}
\end{align*}
where on the last line we bounded the difference of functions by $\mathcal C(f)\delta^{1/2}$ and we used that $\sum_{\mathbf{k} }c_\mathbf{k}=1$.

Thus, it suffices to study the thickening $\int_{W^{\mathrm{u}}(x)}\varphi(y)f(a_t y)\,d\mu^{\mathrm{u}} _x(y)$. Now, we introduce an extra average in the horocycle direction,
$$A = \frac{1}{\tau}\int_{W^{\mathrm{u}}(x)} \int_0^\tau \varphi(y)f(a_tu_ry)\,dr\,d\mu^{\mathrm{u}} _x(y).$$

We will eventually show that $\int_{W^{\mathrm{u}}(x)}\varphi(y)f(a_t y)\,d\mu^{\mathrm{u}} _x(y)$ is comparable to $A$. To see this, notice that $\mu_x ^{\mathrm{u}}\left(u_rB_\mathbf{k}\triangle B_\mathbf{k}\right)\ll \tau\mu_x ^{\mathrm{u}}(B_\mathbf{k})$ since, locally $W^{\mathrm{u}}(x)\simeq \mathbb R^{\text{dim}(W^{\mathrm{u}}(x))}$ and the latter enjoys the Folner property. Hence,
\begin{align*}
    \left|\int_{W^{u}(x)}\varphi(y)f(a_tu_ry)\,d\mu_x ^{\mathrm{u}}(y)-\int_{W^{u}(x)}\varphi(y)f(a_ty)\,d\mu_x ^{\mathrm{u}}(y)\right|&\le \sum_{\mathbf{k}}c_{\mathbf{k}} \int_{u_rB_{\mathbf{k}}\triangle B_{\mathbf{k}}}|f(a_ty)|\,d\mu_x ^{\mathrm{u}}(y)\\
    &\ll \sum_{\mathbf{k}} c_{\mathbf{k}} \tau\mu_x ^{\mathrm{u}}(B_\mathbf{k})\|f\|_\infty\\
    &\ll \mathcal C(f) \tau.
\end{align*}
Integrating the above over $[0,\tau]$ and multiplying the above by $1/\tau$ yields
\begin{equation}\label{approx}
\left|\int_{W^{\mathrm{u}}(x)}\varphi(y)f(a_t y)\,d\mu^{\mathrm{u}} _x(y) - A\right|\ll \mathcal C(f)\tau\end{equation}
 
We choose $\tau$ to be of the form $e^{(\frac{1}{l}-2)t}$ for $l\ge 2.$ Then, by equation (\ref{approx}) and noting that $|\log(\delta)/8|\le t$ we have 
$$\left|\int_{W^{\mathrm{u}}(x)}\varphi(y)f(a_t y)\,d\mu^{\mathrm{u}} _x(y) - A\right|\ll \mathcal C(f)\delta^{1/16},$$
and, as such, we have reduced our analysis to that of $A$.

By the Cauchy-Schwarz inequality, we have
$$|A|^2\le\int_{W^{\mathrm{u}}(x)} \left(\frac{1}{\tau}\int_0 ^\tau f(a_tu_r y)\,dr\right)^2\varphi(y)\,d\mu_x ^{\mathrm{u}}(y).$$

Now that all the terms we are dealing with are non-negative, by utilizing the upper bound on $\varphi$ we deduce
\begin{align}\label{ineq: main}
|A|^2 &\ll \frac{\eta^{-nd}\delta^{-\varepsilon}}{\tau^2} \int_0 ^\tau \int_0 ^\tau\int_{B_1 ^{\mathrm{u}}(x)} \hat f_{r_1,r_2}(a_ty)\,d\mu_x ^{\mathrm{u}}(y)\,dr_1\,dr_2
\end{align}

where $\hat f_{r_1,r_2}(y)=f(a_t u_{r_1}a_{-t}y)f(a_tu_{r_2}a_{-t}y)$.

Observe that
$$\mathcal C(\hat f_{r_1,r_2})\le (e^{2t}\tau)^2\mathcal C(f)^2=e^{\frac{2}{l}t}\mathcal C(f)^2.$$

By the above and by choosing a fixed $l'$  large enough so that $l =4l'\ref{K:effec}^{-1} \ge 2$, we have
$$\mathcal C(\hat f_{r_1,r_2})\le e^{\ref{K:effec}t/2}\mathcal C(f)^2.$$

Combining this observation, Equation (\ref{ineq: main}), and Proposition \ref{Prop: EffUnstable} we deduce
\begin{align*}
\eta^{-nd}\delta ^{-\varepsilon} \left|\int_{B_1 ^{\mathrm{u}}(x)} \hat f_{r_1,r_2}(a_ty)\,d\mu_x ^{\mathrm{u}}(y)\right|&\le \eta^{-nd}\delta ^{-\varepsilon}\int_\mathcal M \hat f_{r_1,r_2}(x)\,d\mu_\mathcal M(x)\\
&+\eta^{-nd}\ref{b:effec}\delta ^{-\varepsilon}e^{-\ref{K:effec}t/2}\mathcal C(f)^2\eta^{-L}.
\end{align*}
Combining the above with Equation (\ref{ineq: main}) yields
\begin{align}\label{ineq: two terms}
    |A|^2\ll\frac{1}{\tau^2}\int_0 ^\tau \int_0 ^\tau \bigg(&\eta^{-nd}\delta ^{-\varepsilon}\int_\mathcal M \hat f_{r_1,r_2}(x)\,d\mu_\mathcal M(x)\\
&+\eta^{-nd}\ref{b:effec}\delta ^{-\varepsilon}e^{-\ref{K:effec}t/2}\mathcal C(f)^2\eta^{-L}\bigg)\,dr_1\,dr_2\nonumber.
\end{align}
Now we analyze the first term on the right of inequality (\ref{ineq: two terms}). We do this by splitting the integral over $[0,\tau]^2$ into two regions; one where $|r_1-r_2|>e^{-2t}e^{t/2l}$ and so we can take advantage of the decay of correlations (Corollary \ref{Cor:Decay of Cor}) and one where $|r_1-r_2|<e^{-2t}e^{t/2l}$ where we use that the region is of small measure. 

When $|r_1-r_2|>e^{-2t}e^{t/2l}$, then by the decay of correlations of Corollary \ref{Cor:Decay of Cor}, and observing that $d(e,u_{e^{2t}(r_1-r_2)})\ge e^{2t}|r_1-r_2|>e^{t/2l}>\frac{t}{2l}$ , we have
\begin{align*}
\int_\mathcal M \hat f_{r_1,r_2}( x)\,d\mu_\mathcal M(x) &\ll e^{-\kappa_\mathcal M e^{2t}|r_1-r_2|} \mathcal C(f)^2\\
 &\le e^{-\kappa_\mathcal M \frac{t}{2l}} \mathcal C(f)^2.
\end{align*}
Thus, after integrating over $[0,\tau]^2$ and dividing by $\tau^2$, we have the following bound on the first term on the right of (\ref{ineq: two terms}) of
$$4^d\ref{b:rich}\delta ^{-\varepsilon}\int_\mathcal M \hat f_{r_1,r_2}( x)\,d\mu_\mathcal M(x)\ll \delta^{-\varepsilon}e^{-\kappa_\mathcal M \frac{t}{2l}} \mathcal C(f)^2$$
whenever $|r_1-r_2|>e^{-2t}e^{\frac{t}{2l}}.$ 


Finally, we consider the region close to the diagonal ($|r_1-r_2|<e^{-2t}e^{t/2l}=\tau e^{-t/2l}$). Notice this region has area $2\tau^2e^{-t/2l}.$ By invariance of $\mu_\mathcal M$ and using that the $\|\cdot\|_\infty$ is dominated by $\mathcal C(\cdot)$ we obtain
$$\int_\mathcal M\hat f_{r_1,r_2}( x) \,d\mu_\mathcal M(x) =\int_\mathcal Mf(u_{e^{2t}(r_1-r_2)} x)f( x)\,d\mu_\mathcal M(x)\le 2\tau^2e^{-t/2l}\|f\|_\infty ^2\le 2\tau^2e^{-t/2l}\mathcal C(f)^2.$$
Integrating over the region $|r_1-r_2|<e^{-2t}e^{t/2l}$ and dividing everything by $\tau^2$ we get that the first term of the right side of inequality (\ref{ineq: two terms}) is smaller than
$$\ll \frac{1}{\tau^2}\mathcal C(f)^2\cdot 2\tau^2e^{-t/2l}\ll \mathcal C(f)^2\cdot e^{-t/2l}.$$
In total, we get this estimate of the right side of inequality (\ref{ineq: two terms}),
\begin{align*}
    |A|^2&\ll \eta^{-L}\mathcal C(f)^2\delta^{-\varepsilon }\left(e^{-\kappa_\mathcal Mt/2l}+e^{-t/2l}+ e^{-\ref{K:effec}t/2}\right)\\&\ll
    \eta^{-L}\mathcal C(f)^2\delta^{-\varepsilon }e^{-\kappa_\mathcal M t/2l}.
\end{align*}

We recall that we have $l=4\ell'\ref{K:effec}^{-1}$ and $e^{-8t}\le\delta$. Let $\eta=e^{-bt}$ for some $b\in(0,1)$. By choosing $$b\le\frac{\kappa_\mathcal M \ref{K:effec}}{16\ell'nd} \text{ and }  \varepsilon\le \frac{\kappa_\mathcal M \ref{K:effec}}{256\ell'}$$  we get 
$$|A|^2\ll  \mathcal C(f)^2 \eta^{-L}e^{-t\kappa_\mathcal M \ref{K:effec}/32l'}$$
and this finishes the proof.
\end{proof}



\end{document}